\documentclass[a4paper]{amsart}
\title[Behavior of corank one singular points]{%
  Behavior of corank one singular points \\
  on wave fronts}
\date{May 18, 2007}
\usepackage[dvips]{graphicx} 
\usepackage{verbatim,enumerate}
\usepackage{amssymb}
\usepackage{amsthm}
\theoremstyle{plain}
 \newtheorem{theorem}{Theorem}[section]
 \newtheorem{maintheorem}{Theorem}
 
 \newtheorem*{theorem*}{Theorem}
 \newtheorem*{lemma*}{Lemma}
 \newtheorem{proposition}[theorem]{Proposition}
 \newtheorem{fact}[theorem]{Fact}
 \newtheorem{fact*}{Fact}
 \newtheorem{lemma}[theorem]{Lemma}
 \newtheorem{corollary}[theorem]{Corollary}
\theoremstyle{remark}
 \newtheorem{definition}[theorem]{Definition}
 \newtheorem{remark}[theorem]{Remark}
 \newtheorem*{remark*}{Remark}
 
 \newtheorem{example}[theorem]{Example}
\numberwithin{equation}{section}
\numberwithin{figure}{section}
\renewcommand{\theenumi}{{\rm(\arabic{enumi})}}
\renewcommand{\labelenumi}{\theenumi}

\newcommand{\pt}[1]{\mathsf{#1}}

\newcommand{\R}{\boldsymbol{R}}

\newcommand{\rank}{\operatorname{rank}}

\renewcommand{\phi}{\varphi}

\newcommand{\sign}{\operatorname{sgn}}

\newcommand{\inner}[2]{\left\langle{#1},{#2}\right\rangle}

\newcommand{\A}{\mathcal{A}}
\newcommand{\E}{\mathcal{E}}

\newcommand{\Sec}{\operatorname{Sec}}

\newcommand{\CE}{\operatorname{C}}
\newcommand{\SW}{\operatorname{SW}}
\newcommand{\CCR}{\operatorname{CCR}}
\newcommand{\DS}{\operatorname{DS}}
\newcommand{\CL}{\operatorname{CL}}
\newcommand{\SB}{\operatorname{SB}}

\author{Kentaro Saji}
\address[Saji]{%
   Department of Mathematics,
   Hokkaido University,
   Sapporo 060-0810,
   Japan
}
\email{saji@math.sci.hokudai.ac.jp}

\author{Masaaki Umehara}
\address[Umehara]{%
   Department of Mathematics, Graduate School of Science,
   Osaka University,
   Toyonaka, Osaka 560-0043,
   Japan
}
\email{umehara@math.wani.osaka-u.ac.jp}
\author{Kotaro Yamada}
\address[Yamada]{%
   Faculty of Mathematics,
   Kyushu University,
   Higashi-ku, Fukuoka 812-8581, Japan%
}
\email{kotaro@math.kyushu-u.ac.jp}
\begin{document}
\begin{abstract}
Let $M^2$ be an oriented $2$-manifold and $f\colon{}M^2\to \R^3$ 
a $C^\infty$-map. 
A point $p\in M^2$ is called a {\em singular point\/} 
if $f$ is not an immersion at $p$.
The map $f$ is called a {\em front\/} (or {\em wave front\/}), 
if there exists a unit $C^\infty$-vector field $\nu$ 
such that the image of each tangent vector $df(X)$ $(X\in TM^2)$
is perpendicular to $\nu$, and the pair $(f,\nu)$ gives
an immersion into $\R^3\times S^2$.
In a previous paper, we gave an intrinsic formulation of wave fronts
in $\R^3$.
In this paper, we shall investigate the behavior of cuspidal edges 
near corank one singular points and 
establish Gauss-Bonnet-type formulas under the intrinsic formulation.
\end{abstract}
\maketitle
\section*{Introduction}
Let $M^2$ be a $2$-manifold and $f:M^2\to \R^3$ a $C^\infty$-map.
A point $p\in M^2$ is called {\em regular\/} if $f$ is an immersion on
a sufficiently small neighborhood of $p$, 
and is called {\em singular\/} if it is not regular. 
To extend the concept of surfaces to a larger class that allows
singularities,
we recall the following definitions:
A $C^\infty$-map $f:M^2\to \R^3$ is called a {\em frontal\/}
if there exists a unit vector field $\nu$ along $f$ such that
$\nu$ is perpendicular to $df(TM^2)$.
By  parallel translations,
$\nu$ can be considered as a map into the unit sphere $S^2$,
which is called the {\em Gauss map\/} of the frontal $f$.
Moreover, if the map
\[
   L:=(f,\nu):M^2\longrightarrow \R^3\times S^2
\]
gives an immersion, 
$f$ is called a {\em front\/} or a {\em wave front}.
Using the canonical inner product on $\R^3$,
we identify the unit tangent bundle
$\R^3\times S^2=T_1\R^3$ with the unit cotangent bundle $T_1^*\R^3$, 
which has the canonical contact structure.
When $f$ is a front, $L$ gives a Legendrian immersion with respect to
the canonical contact structure.
Hence, fronts are considered as projections of Legendrian
immersions.

Consider a front $f\colon{}M^2\to \R^3$ defined on a compact oriented
$2$-manifold $M^2$.
When the set $\Sigma$ of singular points of $f$ consists of cuspidal 
edges and swallowtails,
Langevin-Levitt-Rosenberg \cite{LLR} and Kossowski \cite{K1}
proved the following two Gauss-Bonnet-type formulas
\begin{alignat*}{2}
 2\deg(\nu)&=
        \chi(M_+)-\chi(M_-)
         +\#S_+-\#S_- \qquad &&(\text{\cite{LLR},\cite{K1}}),
 \tag{1}\label{eq:GB-signed}\\
 2\pi\chi(M^2)
         &=\int_{M^2}K\,dA+2\int_{\Sigma}\!\kappa_s\, d\tau
               \qquad &&(\text{\cite{K1}}),
 \tag{2}\label{eq:GB-unsigned}
\end{alignat*}
where $\deg(\nu)$ is the degree of the Gauss map,
$\#S_+$, $\#S_-$ are the numbers of positive and negative swallowtails
respectively,
$M_+$ (resp.\ $M_-$) is the open submanifold of $M^2$
to which the co-orientation is compatible (resp.\ not compatible)
with respect to the orientation,
and $dA$ (resp.\ $d\tau$) is the area element of the surface 
(resp.\ the arclength measure of the singular set).
(See Section~\ref{sec:intrinsic}, or \cite{SUY} for precise definitions.)
The function $\kappa_s$ is called the {\em singular curvature function},
which is originally defined in \cite{SUY}.
In the proofs of these formulas in \cite{LLR} and \cite{K1}, 
the singular curvature implicitly appeared as 
a measure  $\kappa_s\,d\tau$.
(The formula \eqref{eq:GB-signed} is stated in \cite{LLR}, and 
  proofs for both \eqref{eq:GB-signed} and \eqref{eq:GB-unsigned}
  are in \cite{K1}.)

In \cite{SUY}, the authors 
stated a generalization of
\eqref{eq:GB-signed} and \eqref{eq:GB-unsigned} for 
singularities containing double swallowtails,
and gave a sketch of their proofs.

On the other hand, the classical Gauss-Bonnet formula is intrinsic in nature.
So it is quite natural to formulate the singularities 
of fronts intrinsically.
In this paper, we will give a general
setting of intrinsic fronts according to 
the final section of \cite{SUY},
and will prove the intrinsic Gauss-Bonnet formulas 
(Theorem~\ref{thm:B} in Section~\ref{sec:class}).
As a consequence, 
our intrinsic approach also gives a detailed explanation 
of the proofs of (2.2), (2.6) and Theorem 2.3 in \cite{SUY} 
(see Theorem~\ref{thm:A} and Theorem~\ref{thm:B} in Section~\ref{sec:class}).
\section{An intrinsic approach and the singular curvature function}
\label{sec:intrinsic}

In this section, we give a general setting for intrinsic wave
fronts.
\begin{definition}\label{def:coherent}
 Let $M^2$ be an oriented $2$-manifold.
 An orientable vector bundle $\E$ of rank $2$
 with a metric $\inner{~}{~}$ and a metric connection
 $D$ is called an {\em abstract limiting tangent bundle\/} 
 or a {\em coherent tangent bundle\/}
 if there is a bundle homomorphism
 \[
   \psi\colon{}TM^2\longrightarrow \E
 \]
 such that
 \begin{equation}\label{eq:c}
    D_{X}\psi(Y)-D_{Y}\psi(X)=\psi([X,Y]) \qquad (X,Y\in TM^2).
 \end{equation}
\end{definition}

In \cite{SUY}, the authors used the term
an {\em abstract limiting tangent bundle\/}, but
in this paper we shall rather use 
{\em coherent tangent bundle\/}
instead, since it is a shorter phrase.

In this setting, 
the pull-back of the metric $ds^2:=\psi^*\inner{~}{~}$ is called 
{\em the first fundamental form\/} of $\psi$.
A point $p\in M^2$ is called a {\it singular point\/} 
(of $\psi$) if $\psi_p\colon{}T_pM^2\to\E_p$ is not a bijection,
where $\E_p$ is the fiber of $\E$ at $p$, 
that is,  the first fundamental form is not positive definite.
We denote by $\Sigma$ the set of singular points on $M^2$.

Since $\E$ is orientable, there exists a smooth non-vanishing 
skew-symmetric bilinear section 
$\mu\in \Sec(\E^*\wedge\E^*)$
such that $\mu(e_1,e_2)=\pm 1$ for any orthonormal frame $\{e_1,e_2\}$
on $\E$.
The form $\mu$ is determined uniquely up to $\pm$-ambiguity.
A {\em co-orientation\/} of the coherent tangent bundle $\E$ is a 
choice of $\mu$.
A frame $\{e_1,e_2\}$ is called {\em positive\/} 
with respect to the co-orientation $\mu$
if $\mu(e_1,e_2)=+1$.

From now on, we fix a co-orientation $\mu$ on the coherent tangent bundle.
\begin{definition}[Area elements]
 The {\em signed area form\/} $d\hat A$ and the 
 ({\em un-signed}) {\em area form\/} $dA$  are defined 
 on a positively oriented local coordinate system $(U;u,v)$ as
 \begin{multline*}
    d\hat A := \psi^*\mu = \lambda\,du\wedge dv,\quad
     dA     := |\lambda|\,du\wedge dv,\\
  \text{where}\qquad
  \lambda:=\mu\left(
           \psi_u,
           \psi_v
           \right),
         \qquad
         \left(
           \psi_u:=\psi\left(\frac{\partial}{\partial u}\right),
           \psi_v:=\psi\left(\frac{\partial}{\partial v}\right)
         \right).
 \end{multline*}
 We call the function $\lambda$ the {\em signed area density function\/}
 on $U$.
 The set of the singular points on $U$ is expressed as
 \begin{equation}\label{eq:singular}
   \Sigma\cap U:=\{p\in U\,;\,\lambda(p)=0  \}.
 \end{equation}

 Both $d\hat A$  and $dA$ are independent of the
 choice of positively 
 oriented local coordinate system $(u,v)$, and
 give globally defined $2$-forms on $M^2$.
($d\hat A$ is $C^\infty$-differentiable, but $dA$ is only continuous.)
 When $M^2$ has no singular points, the two forms 
 coincide up to sign.
 We set
 \[
   M_+:=\bigl\{p\in M^2\setminus \Sigma \,;\, d\hat A_p=dA_p\bigr\},\qquad
   M_-:=\bigl\{p\in M^2\setminus \Sigma\,;\, d\hat A_p=-dA_p
 \bigr\}.
 \]
The singular set $\Sigma$ coincides with $\partial M_+=\partial M_-$.
\end{definition}
A singular point $p\in\Sigma$ is called {\em non-degenerate\/} if
$d\lambda$ does not vanish at $p$.
On a neighborhood of a non-degenerate singular point,
the singular set consists of a regular curve $\gamma(t)$ on $M^2$, 
called the {\em singular curve}.
The tangential direction of the singular curve is called the 
{\em singular direction}.
If $p$ is a non-degenerate singular point, 
the rank of $\rank\psi_p$ is $1$.
The direction of the kernel of $\psi_p$ is called the 
{\em  null direction}.
Let $\eta(t)$ be the smooth (non-vanishing)
vector field along the singular curve $\gamma(t)$
which gives the null direction.

Here, we give an example:
\begin{example}\label{ex:front}
 Let $M^2$ be an oriented $2$-manifold,
 $f:M^2\to \R^3$ a frontal (see the introduction) and
 $\nu:M^2\to S^2$ its unit normal vector field.
 Then the vector bundle $\E^f$ on  $M^2$
 whose fiber at $p\in M^2$ is given by
 \[
   \E^f_p:=\{X\in T_{f(p)}\R^3\,;\, \mbox{$X$ is perpendicular to $\nu(p)$}\}
 \]
 is called the {\em limiting tangent bundle\/} of  $f$.
 The restriction of the canonical inner product on $\R^3$ 
 gives the metric on $\E^f$, and
 the tangential part of the Levi-Civita connection 
 on $\R^3$ gives the covariant derivative on $\E^f$ satisfying 
 the condition \eqref{eq:c}.
 The bundle homomorphism between $TM^2$ and $\E^f$ is
 given by
 \[
   \psi_p(X):=df_p(X)\in \E^f_p \qquad (X\in T_pM^2,\,\, p\in M^2).
 \]
 We call this $\E^f$ the 
 {\em coherent tangent bundle associated with the frontal $f$}.
 Let $(U;u,v)$ be an arbitrary positively
 oriented local coordinate system of $M^2$. 
 Then the determinant of three vectors
 \[
    \lambda:=\det(f_u,f_v,\nu)
    \qquad\left(
            f_u:=\frac{\partial f}{\partial u},
            f_v:=\frac{\partial f}{\partial v}
           \right)
 \]
 gives the signed area density function on $U$.

 A singular point is called a 
 {\em cuspidal edge\/} or a {\em swallowtail\/} 
 if the corresponding germ of the $C^\infty$-map is 
 $\A$-equivalent to that of the $C^\infty$-map germ
 \begin{equation}\label{eq:cuspidal-swallow}
    f_{\CE}(u,v):=(u^2,u^3,v) \quad\text{or}\quad
    f_{\SW}(u,v):=(3u^4+u^2v,4u^3+2uv,v)
 \end{equation}
 at $(u,v)=(0,0)$, respectively (see Figure~\ref{fig:sw}).
 Here, two $C^{\infty}$-maps 
 $f_i\colon{}U_i\to \R^3$ ($i=1,2$) are {\em $\A$-equivalent\/} 
 (or {\em right-left equivalent}) at the
 points $p_i\in U_i\subset\R^2$ ($i=1,2$) 
 if there exists a local diffeomorphism $\varphi$ of 
 $\R^2$ with $\varphi(p_1)=p_2$ and a local diffeomorphism 
 $\Phi$ of $\R^3$ with $\Phi(f_1(p_1))=f_2(p_2)$
 such that $f_2=\Phi\circ f_1 \circ \varphi^{-1}$.
 It can be easily checked that both $f_{\CE}$ and $f_{\SW}$
 are fronts. 
 These two types of singular points characterize 
 the generic singularities of fronts in $\R^3$.
 The singular curve of $f_{\CE}$ is the $v$-axis and
 the null direction is the $u$-direction.
 The singular curve of $f_{\SW}$ is the parabola $6u^2+v=0$ and
 the null direction is the $u$-direction.
\end{example}

\begin{figure}
 \begin{center}
   \begin{tabular}{c@{\hspace{2cm}}c}
        \includegraphics[height=4.5cm]{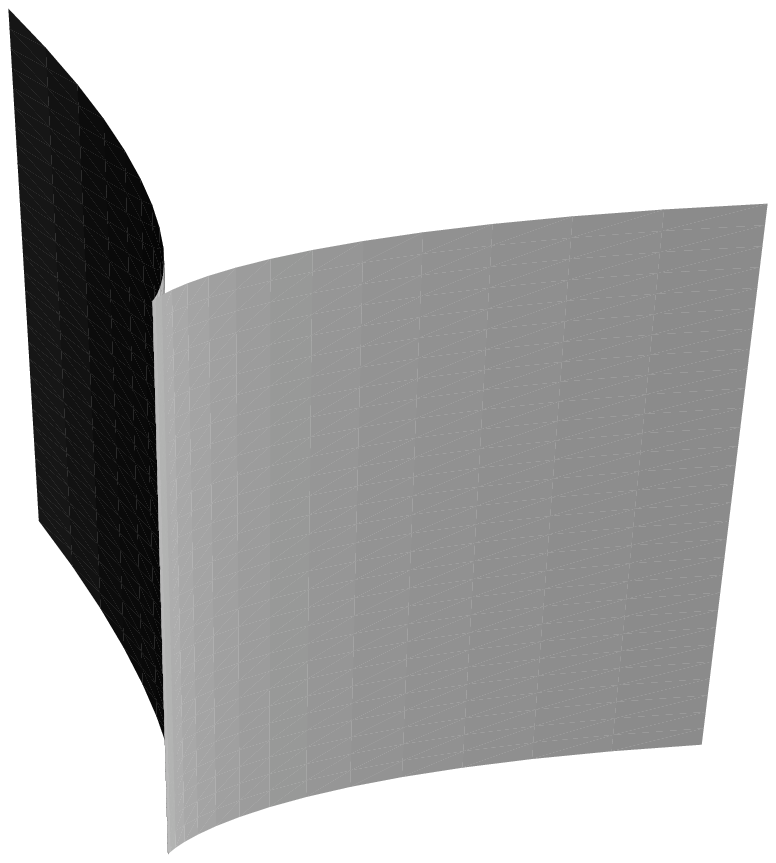} &
        \includegraphics[height=4.5cm]{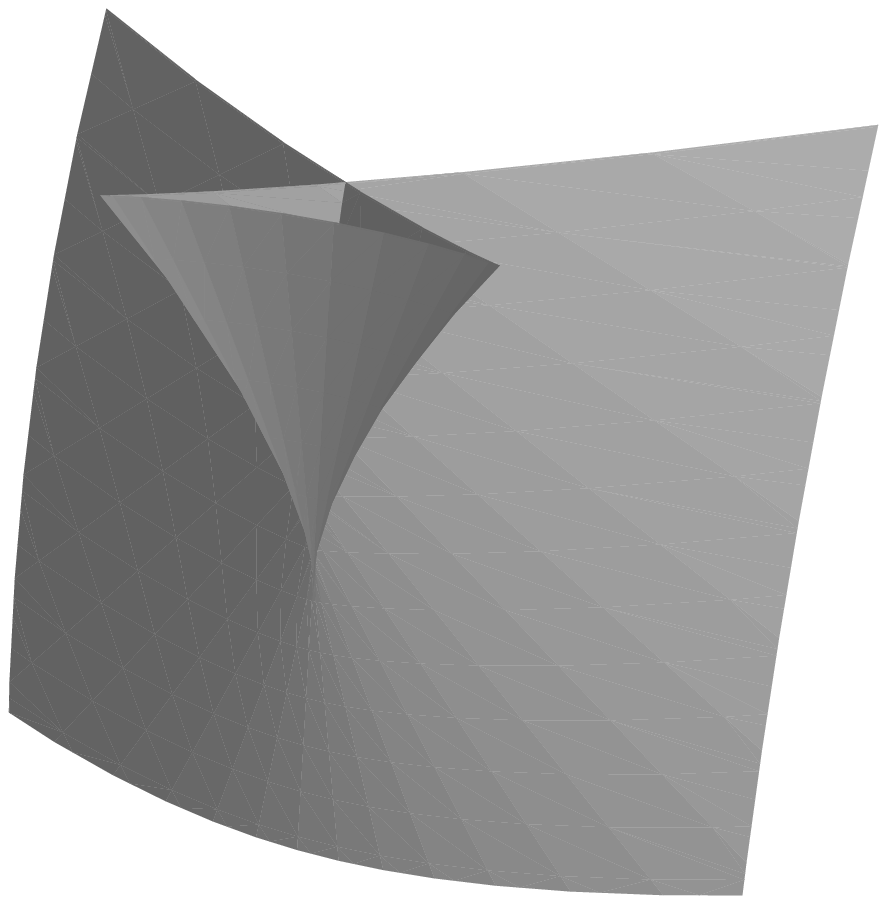}
    \end{tabular}
 \end{center}
\caption{A cuspidal edge and a swallowtail}
\label{fig:sw}
\end{figure}

\begin{definition}\label{def:ak-point}
 Let $(\E,\inner{~}{~},D,\psi)$ be a coherent tangent bundle
 as in Definition~\ref{def:coherent}.
 Take a non-degenerate singular point $p\in M^2$ and 
 let $\gamma(t)$ be the singular curve satisfying  $\gamma(0)=p$.
 Then $p$ is called an {\em $A_2$-point\/} 
 or an {\em intrinsic cuspidal edge\/} if
 the null direction $\eta(0)$ is transversal to
 the singular direction $\dot\gamma(0)=d\gamma/dt|_{t=0}$.
 If $p$ is not an $A_2$-point, but satisfies
 \[
    \left.\frac{d}{dt}\right|_{t=0}
              (\dot\gamma(t)\wedge \eta(t))\ne 0,
 \]
 it is called an {\em $A_3$-point\/} or 
 an {\em intrinsic swallowtail\/}, where $\wedge$ is the
 exterior product on $TM^2$.
\end{definition}

\begin{fact}[\cite{KRSUY}]\label{fact:krsuy}
 Suppose $f:M^2\to \R^3$ is a front.
 Then a non-degenerate singular point $p\in M^2$
 is a cuspidal edge {\rm (}resp.\ swallowtail{\rm )}
 if and only if it is an $A_2$-point {\rm (}resp.\ an $A_3$-point{\rm )}.
\end{fact}

\begin{figure}
\begin{center}
 \includegraphics[width=5cm]{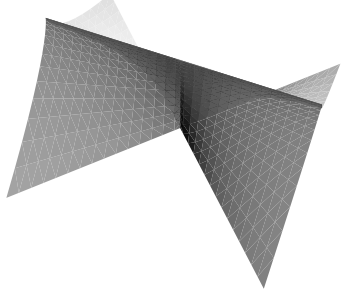}
\end{center}
\caption{A cuspidal crosscap.}\label{fig:ccr}
\end{figure}

\begin{remark}
 A {\em cuspidal cross cap\/} is a singular point which is
 $\A$-equivalent to the $C^\infty$-map germ
 \begin{equation}\label{eq:ccr}
   f_{\CCR}(u,v):= (u,v^2,uv^3),
 \end{equation}
 at $(u,v)=(0,0)$, see Figure~\ref{fig:ccr}.
 The map $f_{\CCR}$ is not a front 
 but a frontal with the unit normal vector field
 \[
   \nu_{\CCR}:=\frac{1}{\sqrt{4+9u^2v^2+4v^6}}
                          (-2v^3,-3uv,2).
 \]
 Though a cuspidal cross cap is different from cuspidal edge, 
 it is also an $A_2$-point in the sense of Definition~\ref{def:ak-point}.
 So Fact~\ref{fact:krsuy} requires the assumption that $f$ is a front.
 (In \cite{FSUY}, a useful criterion for cuspidal cross caps is given.)
\end{remark}

Now we take a coherent tangent bundle $(\E,\inner{~}{~},D,\psi)$
and fix  a singular curve $\gamma(t)$ consisting of $A_2$-points.
Since $d\gamma/dt$ is transversal to the null direction,
the image $\psi(d\gamma/dt)$ does not vanish, and then 
we can take a parameter $\tau$ of $\gamma$ such that 
\[
  \inner{%
      \psi\bigl(\gamma'(\tau)\bigr)}{%
      \psi\bigl(\gamma'(\tau)\bigr)}=1
    \qquad \left('=\frac{d}{d\tau}\right),
\]
which is called the {\em arclength parameter\/} of the
singular curve $\gamma$.
Take a null vector field $\eta(\tau)$ along $\gamma(\tau)$ such that
$\bigl\{\gamma'(\tau),\eta(\tau)\bigr\}$ is a positively oriented frame
field along $\gamma$ for each $\tau$.

Let $n(\tau)$ be a section of $\E$ along $\gamma(\tau)$ such that
$\left\{\psi\bigl(\gamma'(\tau)\bigr),n(\tau)\right\}$ 
is a positive orthonormal frame,
which is called the {\em $\E$-conormal\/} of $\gamma$.
Then
\begin{equation}\label{eq:geodesic-curvature}
  \hat\kappa_g(\tau):=
      \inner{D_\tau\psi\bigl(\gamma'(\tau)\bigr)}{n(\tau)} 
         =\mu\left(
            \psi\bigl(\gamma'(\tau)\bigr),
                D_\tau\psi\bigl(\gamma'(\tau)\bigr)
              \right)
\end{equation}
is called the {\em $\E$-geodesic curvature\/} of $\gamma$,
which gives the geodesic curvature of the singular curve $\gamma$ 
with respect to the orientation of $\E$,
where $D_\tau=D_{d/d\tau}$.
Then the {\em singular curvature function\/} is defined by 
\begin{equation}\label{eq:sing-curve}
  \kappa_s(\tau):=\sign\bigl(d\lambda(\eta(\tau))\bigr)\hat \kappa_g(\tau),
\end{equation}
where 
$\sign\bigl(d\lambda(\eta(\tau))\bigr)$ denotes the sign of the function
$d\lambda(\eta)$ at $\tau$.
In a general parametrization of $\gamma=\gamma(t)$, the singular
curvature function is
\begin{equation}\label{eq:sing-curve-general}
 \kappa_s(t) = \sign\bigl(d\lambda(\eta(t))\bigr)
          \frac{
           \mu\bigl(\psi(\dot\gamma(t)),D_t\psi(\dot\gamma(t))\bigr)}{
       |\psi(\dot\gamma(t))|^3}
\qquad \left(\dot{~}=\frac{d}{dt}\right),
\end{equation}
where $|\xi|:=\sqrt{\inner{\xi}{\xi}}$ denotes the norm derived from 
the metric $\inner{~}{~}$.

\begin{proposition}[An intrinsic version of Theorem 1.6 in \cite{SUY}]
 The singular curvature function does not depend on 
 the orientation of $M^2$,  
 nor the orientation of $\E$,
 nor the parameter $t$ of the singular curve $\gamma(t)$.
\end{proposition}
\begin{proof}
 If the orientation of $M^2$ reverses, then $\lambda$ and $\eta$
 both change sign.
 If the orientation of $\E$ (i.e.\ the co-orientation) reverses, 
 then $\lambda$ and the $\E$-conormal $n$  both change sign.
 If $\gamma$ changes orientation, both $\gamma'$ and $\eta$ 
 change sign. 
 In all cases, the sign of $\kappa_s$ is unchanged.
\end{proof}

Let $\{e_1,e_2\}$ be a positive orthonormal frame field of $\E$ defined
on a domain $U\subset M^2$.
Then there exists a unique $1$-form $\omega$ on $U$ such that
\[
   D_Xe_1=-\omega(X)e_2,
   \qquad 
   D_Xe_2=\omega(X)e_1\qquad (X\in TU),
\]
which is called the {\em connection form\/} 
with respect to the frame $\{e_1,e_2\}$.
The exterior derivative $d\omega$ does not
depend on the choice of a positive frame $\{e_1,e_2\}$
and gives a (globally defined) $2$-form on $M^2$.
By the definition of $\omega$ we have
\begin{equation}\label{eq:d-omega}
   d\omega
        =K\, d\hat A=\begin{cases}
            \hphantom{-}K\, dA \qquad& \mbox{on $M_+$}, \\
             -K\, dA           \qquad& \mbox{on $M_-$}, 
           \end{cases}
\end{equation}
where $K$ is the Gaussian curvature of the first fundamental form
$ds^2$.
When $M^2$ is compact, the integration
\begin{equation}\label{eq:Euler}
 \chi^{}_{\E}:=\frac{1}{2\pi}\int_{M^2}K\, d\hat A=
            \frac{1}{2\pi}\int_{M^2}d\omega
\end{equation}
is an integer called the {\em Euler number\/} of $\E$.
\section{Peaks and the interior angles between singular curves}
\label{sec:class}

To formulate our generalized Gauss-Bonnet formula,
we define further singularities 
(which is the essentially the same definition as in \cite{SUY}):
\begin{definition}[Peaks]\label{def:peak}
 A singular point $p\in M^2$ (which is not an $A_2$-point)
 is called a {\em peak\/} 
 if there exists a coordinate neighborhood $(U;u,v)$ of $p$
 such that
 \begin{enumerate}
  \item\label{item:peak-1} 
       there are no singular points other than $A_2$-points on 
       $U\setminus \{p\}$,
  \item\label{item:peak-2} 
       the rank of the linear map $\psi_p\colon{}T_pM^2\to \E_p$
       at $p$ is equal to $1$, and
  \item\label{item:peak-3} 
       the singular set in $U$ consists
       of finitely many (possibly empty) $C^1$-regular curves 
       starting from $p$.
       (If such a set of regular curves is empty, 
       the peak $p$ is an isolated singular point.)
 \end{enumerate}
 If a peak is a non-degenerate singular point, it
 is called a {\em non-degenerate peak}.
 The singular set $\Sigma$ is said to {\em admit at most peaks\/}
 if it consists of $A_2$-points and peaks.
\end{definition}

Let $U$ be a sufficiently small neighborhood of a peak $p$ and 
$\sigma_1,\sigma_2$ two singular curves in  $U$ starting at $p$.
A domain $\Omega$ satisfying the following two conditions is
called a {\em singular sector\/} at $p$:
\begin{enumerate}
 \item The boundary of $\Omega\cap U$ consists of $\sigma_1,\sigma_2$
       and the boundary of $U$. 
 \item There are no singular points in $\Omega$.
\end{enumerate}
If the peak $p$ is isolated,
we also call the domain $\Omega=U\setminus\{p\}$ a singular
sector.
If $\Omega$ is a singular sector at $p$, the whole of $\Omega$
is contained in $M_+$ or $M_-$.
When $\Omega\subset M_+$ (resp. $\Omega\subset M_-$), it is called a 
{\em positive\/} ({\em negative}) singular sector.
If the number of singular sectors are more than two, 
the number of positive sectors is
equal to the number of negative sectors at each peak.

Swallowtails (or more generally $A_3$-points) 
are  examples of non-degenerate peaks, which have
two singular sectors.
There are singular points which are not peaks.
Typical examples are cone-like singularities 
which appear in rotationally symmetric surfaces in $\R^3$ of positive
constant Gaussian curvature. 
However, since  generic fronts (in the local sense) have only cuspidal edges
and swallowtails,
the set of fronts which admit at most peaks covers a sufficiently wide
class of fronts.

\begin{figure}
 \begin{center}
  \includegraphics{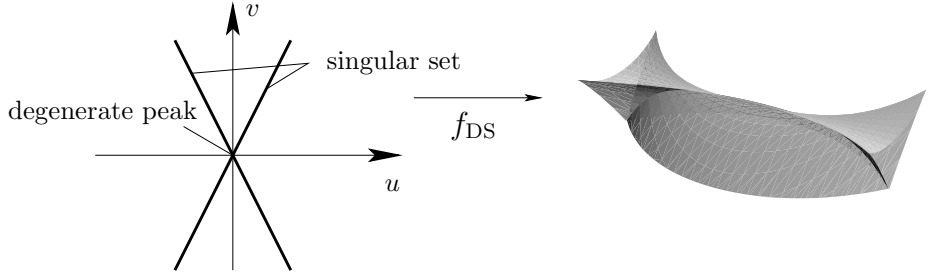}
 \end{center}
 \caption{A double swallowtail (Example~\ref{ex:double-swallow})}
 \label{fig:double-swallow}
\end{figure}
\begin{example}[A double swallowtail, {\cite[Example 1.11]{SUY}}]
 \label{ex:double-swallow}
 A {\em double swallowtail\/} (or a {\em cuspidal beaks})
 is a singular point which is 
 $\A$-equivalent to the $C^\infty$-map germ
 \[
     f_{{\DS}}(u,v) := (2u^3-uv^2,3u^4-u^2v^2,v)
 \]
 at $(u,v)=(0,0)$,  see Figure~\ref{fig:double-swallow}.
 Then 
 \[
    \nu_{\DS}^{} = \frac{1}{\sqrt{1+4u^2(1+u^2v^2)}}(-2u,1,-2u^2v)
 \]
 gives the unit normal vector of $f_{\DS}$.
 It can be easily checked that $f_{\DS}$ is a front.
 The signed area density function is 
 $\lambda = (v^2-6u^2)\sqrt{1+4u^2(1+u^2v^2)}$, and 
 then the singular set is 
 $\Sigma=\bigl\{v=\sqrt{6}u\bigr\}\cup \bigl\{v=-\sqrt{6}u\bigr\}$.
 In particular, the origin is a degenerate peak, whose neighborhood
 is divided into four singular sectors (two of them are positive).
\end{example}

\begin{figure}
 \begin{center}\footnotesize
  \begin{tabular}{c@{\hspace{1cm}}c}
   \includegraphics[width=5.5cm]{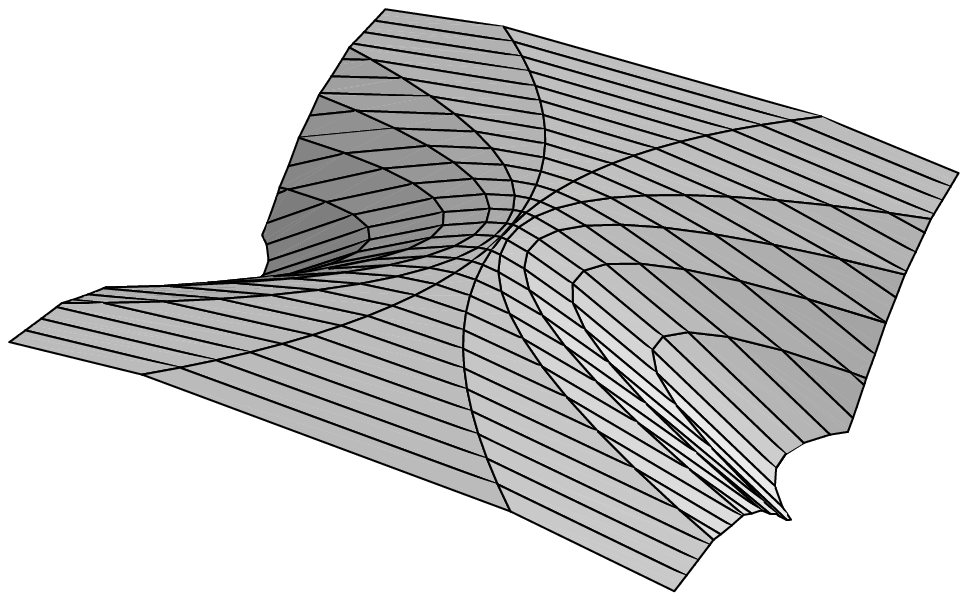}
   &
   \includegraphics[width=5.5cm]{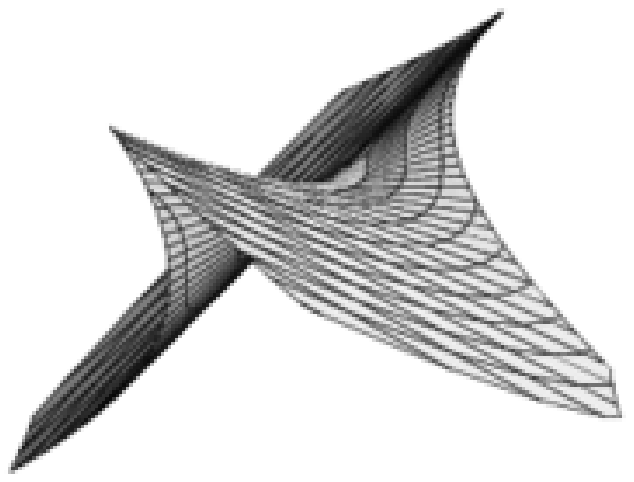}\\
    A cuspidal lips &
    Example~\ref{ex:nondeg} \\
    (Example~\ref{ex:cuspidal-lip}) 
   &
  \end{tabular}
 \end{center}
  \caption{A cuspidal lips and Example~\ref{ex:nondeg}.}
  \label{fig:lip-and-nondeg}
\end{figure}

\begin{example}[A cuspidal lips]
 \label{ex:cuspidal-lip}
 A {\em cuspidal lips\/} is a singular point which is defined by
 \[
     f_{{\CL}}(u,v)=(u^3+uv^2,3u^4+u^2v^2,v),
 \]
 see Figure \ref{fig:lip-and-nondeg}, left-hand side.
 Then
 \[
     \nu_{{\CL}}(u,v)=\frac{1}{\sqrt{1+u^2+16u^4v^2}}(-u,-1,4u^2v)
 \]
 gives the unit normal vector of $f_{\CL}$.
 It can be easily checked that $f_{\CL}$ is a front.
 The singular set is the origin, this is an example of
 degenerate peak without singular $A_2$-curves.
 (In \cite{IST}, useful criteria for 
  cuspidal lips and beaks are given.)
\end{example}

\begin{example}\label{ex:nondeg}
  The tangential developable of the space curve
 $t\mapsto (t^3,t^4,t^5)$ is given by
  \[
    f(t,u)  :=  (t^3+3u,t^4+4tu,t^5+5t^2u).
  \]
 Then $(0,0)$ is a non-degenerate peak
 which is not an $A_3$-point.
 See Figure~\ref{fig:lip-and-nondeg}, right-hand side.
 Ishikawa \cite{I} showed that
 the tangential developables of the space curves
 of the form
 \[
 \gamma(t)=\bigl(t^3 a(t),t^4b(t),t^5c(t)\bigr)
    \qquad \bigl(a(0)b(0)c(0)\ne 0\bigr)
 \]
 at $t=0$ are $\mathcal A$-equivalent to this
 example, where $a(t),b(t),c(t)$ are $C^\infty$-functions.
\end{example}

\begin{example}[The Scherbak surface]
 \label{ex:scherbak}
 The {\em Scherbak surface\/} is a singular point which is defined by
 \[
    f_{{\SB}}(u,v)=(u^3+u^2v,6u^5+5u^4v,v),
 \]
 see Figure \ref{fig:scherbak}.
 Then 
 \[
    \nu_{{\SB}}(u,v)=\frac{1}{\sqrt{1+100u^4+25u^8}}(10u^2,-1,-5u^4)
 \]
 gives the unit normal vector of $f_{{\SB}}$.
 It can be easily checked that $f_{{\SB}}$ is a frontal 
 (see the introduction for the definition).
 The singular set is two transversal lines
 $\{u=0\}\cup\{3u+2v=0\}$.
 The Scherbak surface is investigated in \cite{S,I,IC}.
\end{example}

\begin{figure}
 \begin{center}\footnotesize
   \includegraphics[width=5.5cm]{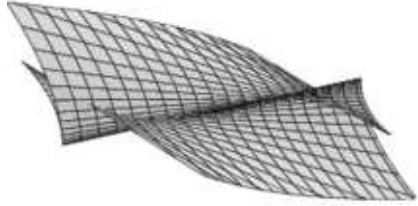}
 \end{center}
  \caption{The Scherbak surface (Example~\ref{ex:scherbak}).}
  \label{fig:scherbak}
\end{figure}

In this section, we fix a co-oriented
coherent tangent bundle $(\E,\inner{~}{~},D,\psi)$
on an oriented manifold $M^2$.
{\em Throughout this section, we assume the singular set
    $\Sigma$ consists of at most peak singularities.}

In the following discussions, we fix an arbitrary
Riemannian metric $g$ on $M^2$.
Since the first fundamental form $ds^2=\psi^*\inner{~}{~}$ 
degenerates on $\Sigma$,
it is useful to use such a metric $g$ to investigate the property of
the singular set $\Sigma$.
Then there exists a $(1,1)$-tensor field $I$ on $M^2$ such that
\[
    ds^2(X,Y)=g(IX,Y)\qquad (X,Y\in T_pM^2, p\in M^2).
\]
We fix a singular point $p\in \Sigma$.
Since $\Sigma$ only admits at most peaks, the kernel of $\psi_p$ is one
dimensional.
Thus only one of the eigenvalues of $I_{p}:T_{p}M^2\to T_{p}M^2$
vanishes.
So there exists a neighborhood $V$ of $p$ such that
$I_q$ has two distinct eigenvalues $0\le \lambda_1(q)< \lambda_2(q)$
for each $q\in V$. 
Since the eigenvectors of these two eigenvalues
$\lambda_1(q)$, $\lambda_2(q)$ depend smoothly on $q\in V$,
there exists a coordinate neighborhood $(U;u,v)$ of $p$
such that $U\subset V$ and 
the $u$-curves (resp.\ the $v$-curves) give the
$\lambda_1$-eigendirections 
(resp.\ the $\lambda_2$-eigendirections) of $I$ on $U$.
We call such a local coordinate system $(U;u,v)$
a {\em $g$-coordinate system\/} at the singular point $p$.

\begin{proposition}\label{prop:limit}
 Let $(U;u,v)$ be a $g$-coordinate system at a peak $p$, and
 $\gamma(t)$ $(0\le t<1)$ a $C^1$-regular curve on $U$
 emanating from $p$ such that
 \begin{enumerate}
  \item $\dot \gamma(0)=d\gamma/dt|_{t=0}$ is not a null-vector, or
  \item $\gamma$ is a singular curve,
 \end{enumerate} 
 Then there exists a limit
 \[
     \Psi_\gamma:=\lim_{t\to +0}
        \frac{\psi\bigl(\dot \gamma(t)\bigr)}{%
             |\psi\bigl(\dot \gamma(t)\bigr)|} 
              \in \E_p.
 \]
\end{proposition}
We call this limit vector $\Psi_\gamma$ 
the {\em $\E$-initial vector\/} of $\gamma$ at $p$.
\begin{proof}
 If $\dot \gamma(0)$ is not a null-vector, the assertion is
 obvious. 
 So we may assume that $\gamma$ is a singular curve such that 
 $\dot \gamma(0)$ is a null-vector.
 We fix a $g$-coordinate system  $(U;u,v)$ at the peak $p$, and 
 write
 \[
    \gamma(t)=\bigl(u(t),v(t)\bigr)\qquad (0\le t<1).
 \]
 Since $\gamma$ is a singular curve, $(\psi_u:=)\psi(\partial/\partial u)$
 vanishes on $\gamma$.
 So we have
 \[
    \psi\bigl(\dot \gamma(t)\bigr)=
        \dot v \,\psi_v\bigl(\gamma(t)\bigr)
          \qquad (0<t<1) ,
 \]
 where
 $\psi_v(q):=\psi_{q}(\partial/\partial v)$ for $q\in U$.
 Since $\psi_v(p)\ne 0$ by the definition of the $g$-coordinate system,
 we have
 \[
    \lim_{t\to +0}\frac{\psi(\dot \gamma)}{|\psi(\dot \gamma)|}
      =\left(\lim_{t\to+ 0} \sign \dot v\right)
           \frac{\psi_v(p)}{|\psi_v(p)|}.
 \]
 Since $\gamma(t)$ ($t>0$) consists of $A_2$-points,
 we have $\dot v\ne 0$ (see Proposition \ref{thm:class}). 
 Hence the sign of $\dot v(t)$ never changes on $t>0$
 and then the limit of $\sign(\dot v)$  exists,
 which proves the assertion.
\end{proof}

\begin{definition}\label{def:angle}
 Let $(U;u,v)$ be a local coordinate system centered at a peak $p$ and
 $\gamma_j(t)$ ($0\le t<1$, $j=1,2$) two $C^1$-regular curves 
 in $U$ emanating from $p$ satisfying the assumption of 
 Proposition~\ref{prop:limit}.
 (We might not choose $(u,v)$ to be a $g$-coordinate system here.)
 Then the angle 
 \[
     \arccos\left(\inner{\Psi_{\gamma_1}}{\Psi_{\gamma_2}}
                         \right)\in [0,\pi]
 \]
 is called the {\em angle between the initial vectors\/} 
 of $\gamma_1$, $\gamma_2$.
\end{definition}
Now, we define the interior angle of a singular sector.
While it may take a value greater than $\pi$, we have to 
divide the singular sector into subsectors such that 
the ``interior angle'' does not exceed $\pi$.

First, we assume that $\Omega$ is  bounded by two singular curves
$\sigma_0$ and $\sigma_1$.
Then there exist a positive integer $n$ and a sequence of 
$C^1$-regular curves starting at $p$
\[
  \sigma_0=\gamma_0,\quad
  \gamma_{1},\quad\dots,\quad\gamma_{n}=\sigma_1  
\]
in $\Omega$ 
satisfying the assumption of Proposition~\ref{prop:limit} 
such that 
\begin{align}
 &\text{
   \begin{minipage}[t]{0.8\linewidth}
     $\gamma_0$, \dots, $\gamma_n$ do not intersect
     each other in $\Omega$.
   \end{minipage}
  }\label{eq:interpol-0}\\
 &\text{
   \begin{minipage}[t]{0.8\linewidth}
      For each $j=1,\dots,n$, there exists a sector domain 
      $\omega_j\subset \Omega$ bounded by $\gamma_{j-1}$ and 
      $\gamma_j$ which does not intersect 
      $\gamma_k$  for each $k\neq j-1,j$.
   \end{minipage}
  }\label{eq:interpol-1}\\
 &\text{
   \begin{minipage}[t]{0.8\linewidth}
     If $n\geq 2$, 
     $\{\dot\gamma_{j-1}(0),\dot\gamma_j(0)\}$
     is linearly independent and positively oriented
     for each $j=1$, \dots, $n$.
   \end{minipage}
  }\label{eq:interpol-2}
\end{align}
  In the following Remark~\ref{rem:explicit},
  we give an explicit way to find $\{\gamma_j\}$.

Next, we assume that the peak $p$ is an isolated singular point.
In this case, there are no singular curves which bound the sector
$\Omega$, but we can take a sequence 
$\{\gamma_0,\gamma_1,\gamma_2\}$ satisfying 
\eqref{eq:interpol-0}--\eqref{eq:interpol-2}, and 
we set $\gamma_3=\gamma_0$.
See Case 6 in Remark~\ref{rem:explicit}.

In both cases, 
the {\em interior angle of the singular sector} $\Omega$ is 
defined as 
\begin{equation}\label{eq:interior-angle}
  \arccos\left(\inner{\Psi_{\gamma_0}}{\Psi_{\gamma_1}}\right)
 +\arccos\left(\inner{\Psi_{\gamma_1}}{\Psi_{\gamma_2}}\right)
 +\cdots+
 \arccos\left(\inner{\Psi_{\gamma_{n-1}}}{\Psi_{\gamma_{n}}}\right).
\end{equation}

\begin{figure}
\begin{center}
\footnotesize
 \begin{tabular}{ccc}
  \includegraphics{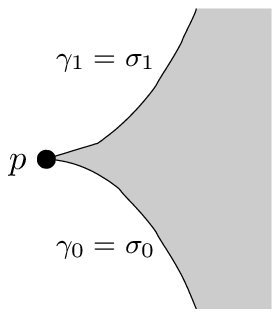} &
  \includegraphics{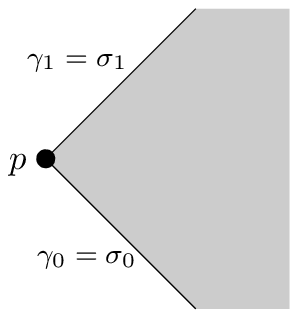} &
  \includegraphics{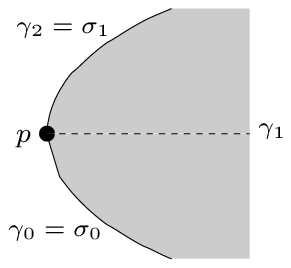} \\
  Case 1 & Case 2 & Case 3 \\[5mm]
  \includegraphics{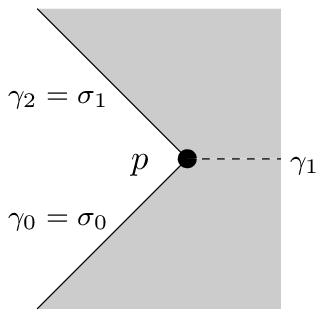} &
  \includegraphics{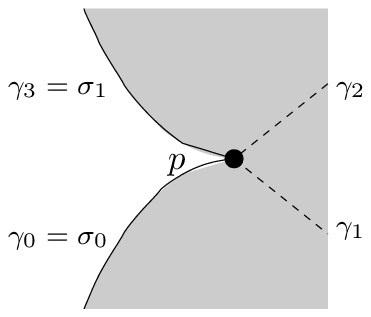} &
  \includegraphics{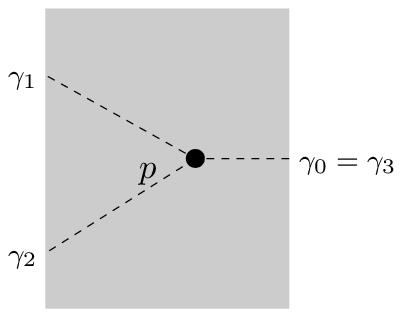} \\
  Case 4 & Case 5 & Case 6 
 \end{tabular}
\end{center}
  \caption{Six possibilities for $\Omega$}
  \label{fig:omega}
\end{figure}
\begin{remark}\label{rem:explicit}
 There are six possibilities
 for $\Omega$ as in  Figure \ref{fig:omega}.
 \begin{description}
  \item[Case 1]
     $\dot\sigma_0(0)=k\dot\sigma_1(0)$ where $k>0$
     and
     $\Omega$ does not contain
     the direction of $-\dot\sigma_0(0)$.
     In this case, we cannot take any interpolation, that is,
     we must take $n=1$.
  \item[Case 2]
     $\dot\sigma_0(0)$ and $\dot\sigma_1(0)$
     are transversal
     and
     $\Omega$ does not contain the directions of
     $-\dot\sigma_0(0)$ and
     $-\dot\sigma_1(0)$.
     In this case, we do not need interpolation,
     that is, we may take $n=1$.
 \item[Case 3]
     $\dot\sigma_0(0)=-k\dot\sigma_1(0)$ where $k>0$.
     In this case, we need
     an interpolation, namely, we may take $n=2$.
 \item[Case 4]
     $\dot\sigma_0(0)$ and $\dot\sigma_1(0)$
     are transversal and
     $\Omega$ contains the directions of
     $-\dot\sigma_0(0)$ and
     $-\dot\sigma_1(0)$.
     In this case, we need
     an interpolation, namely, we may take $n=2$.
 \item[Case 5]
     $\dot\sigma_0(0)=k\dot\sigma_1(0)$, where $k>0$
     and $\Omega$
     contains the directions of $-\dot\sigma_0(0)$.
     In this case, we need two interpolations, namely, we may take
     $n=3$.
 \item[Case 6] The peak is an isolated singular point.
      In this case, we need three curves.
 \end{description}
\end{remark}

Later at the end of this section, we shall prove the following:
\begin{maintheorem}\label{thm:A}
 Let $p\in M^2$ be a peak of a coherent tangent bundle
 $(\E,\inner{~}{~},D,\psi)$ on $M^2$. 
 Then the sum $\alpha_+(p)$ {\rm(}resp.\ $\alpha_-(p)${\rm)}
 of all interior angles of positive {\rm(}resp.\ negative{\rm)}
 singular sectors  at $p$ satisfies 
 \begin{align}
  \label{eq:sum}    &\alpha_+(p)+\alpha_-(p)=2\pi, \\
  \label{eq:diff}   &\alpha_+(p)-\alpha_-(p)\in \{2\pi,0,-2\pi\}.
 \end{align}
\end{maintheorem}
\begin{definition}\label{def:positive}
 A peak $p$ is called {\em positive}, {\em null}, {\em negative}
 according to the sign of $\alpha_+(p)-\alpha_-(p)$.
\end{definition}

\begin{remark}
 The formulas
 \eqref{eq:sum} and  \eqref{eq:diff}  are intrinsic versions
of (2.2) and (2.3) in \cite{SUY} respectively. 
\end{remark}

For our further analysis of the singular curvature near a peak,
we prepare the following assertion which is the
intrinsic version of \cite[Proposition 1.12]{SUY}.
\begin{proposition}[Boundedness of the singular curvature measure]
\label{prop:bounded-curvature-measure}
 Take  a singular curve $\gamma\colon{}[0,\varepsilon)\to M^2$
 starting from a peak $p$ such that  $\gamma(t)$ is a $A_2$-point 
 for each $t>0$. 
 Then the singular curvature measure $\kappa_s\,d\tau$ is continuous
 on $[0,\varepsilon)$, where $d\tau$ is the arclength measure with 
 respect to the first fundamental form $ds^2$.
\end{proposition}
\begin{proof}
 We can take a $g$-coordinate system $(u,v)$ such that 
 $\partial/\partial u$  is the null vector field on $\gamma$.
 For the sake of simplicity, we set
 \[
   \psi_u:=\psi(\partial/\partial u),\quad
   \psi_v:=\psi(\partial/\partial v), 
   \quad \dot {\hat \gamma}(t):=\psi(\dot \gamma(t))
\qquad
    \left(\dot{~}=\frac{d}{dt}\right)
 \]
 In such a coordinate system,  $\psi_u=0$ and $D_t \psi_u=0$ hold on
 $\gamma$.
 Then 
 \[
    \dot{\hat\gamma} = \dot v\psi_v,\qquad
    D_t\dot {\hat\gamma} = \ddot v \psi_v + \dot v D_t \psi_v.
 \]
 Hence 
 \begin{equation}\label{eq:singular-curvature-peak}
    \kappa_s =
    \pm \frac{\mu(\dot{\hat\gamma},D_t\dot{\hat\gamma})}{|\dot{\hat\gamma}|^3}
     = \pm \frac{\mu(\psi_v,D_t \psi_v)}{|\dot v|\,|\psi_v|^3}.
 \end{equation}
 Since $d\tau=|\dot{\hat\gamma}|\,dt=|\dot v|\,|\psi_v|\,dt$ and 
 $\psi_v\neq 0$,
 \[
    \kappa_s\,d\tau =\pm \frac{\mu(\psi_v,D_t\psi_v)}{|\psi_v|^2}\,dt
 \]
 is bounded.
\end{proof}

Then we can state  generalized Gauss-Bonnet formulas:
\begin{maintheorem}
\label{thm:B}
 Let $M^2$ be a compact oriented $2$-manifold and
 $\E$ a coherent tangent bundle whose singular set $\Sigma$
 admits at most peaks.
 Then 
 \begin{align}
  (\chi_{\E}^{}=)&\frac1{2\pi}\int_{M^2}K\, d\hat A=
        \chi(M_+)-\chi(M_-)+\#P_+-\#P_-, 
  \label{eq:B}\\
  2\pi\chi(M^2)&=
  \int_{M^2}K\, dA+2\int_{\Sigma} \kappa_s\, d\tau 
  \label{eq:A}
 \end{align}
 hold, where $d\tau$ is the arclength measure on the singular set
 and $\#P_+,\#P_-$ are the numbers of positive and negative 
 peaks respectively defined in Definition~\ref{def:positive}.
\end{maintheorem}
The identity \eqref{eq:B} and \eqref{eq:A} are generalizations 
of \eqref{eq:GB-signed} and \eqref{eq:GB-unsigned} in the introduction,
respectively.
The proof is given in Section~\ref{sec:globalGB}.

It should be remarked that the integral $\int_{\Sigma} \kappa_s\, d\tau$
is well-defined by Proposition~\ref{prop:bounded-curvature-measure}.
In \cite{SUY}, the authors did not state  Theorem~\ref{thm:B}
intrinsically as above.
The two formulas \eqref{eq:B} and \eqref{eq:A} are not only
generalizations of the formulas given in the introduction, 
but also those in \cite{SUY}.

To prove Theorem~\ref{thm:A}, we need a tool to measure 
the interior angle between ``curves'' starting at a peak.
We define a class of curves such that the interior angles
are well-defined:
\begin{definition}[Admissible curves]\label{def:admissible-curve}
 A curve $\sigma(t)$ $(t\in [a,b])$ on $U$
 is called {\em admissible\/} if it satisfies one of the 
 following conditions:
 \begin{enumerate}
  \item $\sigma$ is a $C^1$-regular curve such that 
    $\sigma\bigl((a,b)\bigr)$ does  not contain a  peak,
    and the tangent vector $\dot\sigma(t)$ ($t\in [a,b]$)
    is transversal to the singular direction and the null direction
    if $\sigma(t)\in\Sigma$.
  \item The set $\sigma([a,b])$ is contained in a singular set $\Sigma$
    and the set $\sigma\bigl((a,b)\bigr)$ does not contain a peak. 
 \end{enumerate}
\end{definition}

Next, we shall prove the following assertion, 
which will play a crucial role in the proof of Theorem~\ref{thm:A}.

\begin{proposition}\label{thm:class}%
 Suppose that $p$ is a peak. 
 Then there exists a $g$-coordinate system 
 $(U;u,v)$ such that each admissible curve $\gamma(t)$ 
 {\rm (}$t\ge 0${\rm )}
 starting at $p$ does not have velocity vector $\dot\gamma(t)$
 parallel to the $u$-axis on $U$\!. 
 In particular,
 $\gamma(t)$ {\rm (}$t> 0${\rm )} never meets the $u$-axis.
\end{proposition}

\begin{proof}
 We fix a $g$-coordinate system $(U;u,v)$ at $p$.
 By definition, any admissible curve in $U$ which is not
 a singular curve never meets the $u$-axis, by the mean value theorem.
 So it is sufficient to consider only singular curves. 
 We now fix  a singular curve $\gamma(t)$ ($t\ge 0$) on $U$
 such that $\gamma(0)=p$.
 Since the number of singular curves starting at $p$ is finite, 
 it is sufficient to show that there exists a 
 (sufficiently small)  $\varepsilon>0$
 such that $\dot\gamma(t)$ will never be parallel to the $u$-axis on  
 $(0,\varepsilon]$.
 (The second assertion immediately follows from the mean value theorem.) 
 If $\dot\gamma(0)$ is transversal to the $u$-axis,
 it is obvious. 
 So we may assume that  $\dot\gamma(0)$ is proportional to the 
 $u$-axis.
 Then $\gamma$ can be expressed as a graph $v=F(u)$.
 If there exists $c\in (0,\delta)$ such that
 $dF/du$ vanishes at $u=c$, 
 then we have a contradiction
 since the null direction $\partial/\partial u$ is
 proportional to the singular direction.
 (On a $g$-coordinate system, the null direction always
 points in the $u$-direction on each singular curve, by its definition.)
\end{proof}

\begin{figure}
 \begin{center}
  \includegraphics{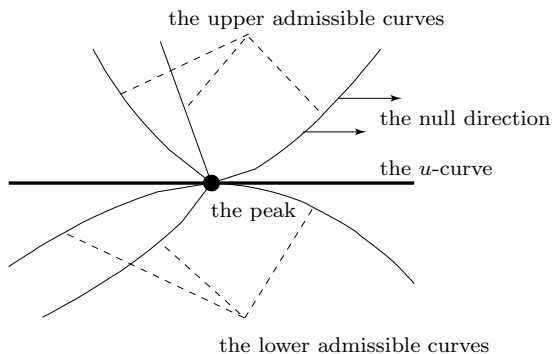}
 \end{center}
 \caption{Upper and lower admissible curves.}
 \label{fig:upper-lower}
\end{figure}

We divide the set of admissible curves
starting at the peak $p$ into the following two classes 
(see Figure~\ref{fig:upper-lower}):
\begin{itemize}
 \item The admissible curves which lie upper half-plane of
       the $g$-coordinate system are called the 
       {\em upper admissible curves},
 \item The admissible curves which lie lower half-plane 
       of the $g$-coordinate system are called the 
       {\em lower admissible curves}.
\end{itemize}

\begin{proposition}
 These two classes of the admissible curves starting from $p$
 are independent of the choice of Riemannian metric $g$ on $M^2$.
\end{proposition}
\begin{proof}
 Let $h$ be another Riemannian metric on $M^2$.
 Then $(1-t)g+th$ ($t\in [0,1]$)
 gives the deformation between two metrics.
 During the deformation of the metric,
 each admissible curve never meets the $u$-axis
 for a fixed sufficiently small neighborhood of $p$.
\end{proof}

Note that admissible curves have initial vectors,
by Proposition~\ref{prop:limit}.
Now, we shall prove that the angle between
two admissible curves 
at a peak are determined by their classes:

\begin{proposition}\label{prop:g-angle}
 Let $\gamma_j(t)$ $(j=0,1)$ be two admissible curves
 starting at a peak $p$.
 Then the $\E$-initial vector $\Psi_{\gamma_1}$ is 
 equal to $\Psi_{\gamma_2}$ {\rm(}resp.\ $-\Psi_{\gamma_2}${\rm )}
 if and only if $\{\gamma_1,\gamma_2\}$ are in the
 same class {\rm(}resp.\ distinct classes{\rm)}.
\end{proposition}

\begin{proof}
 Take a $g$-coordinate system $(U:u,v)$ at $p$.
 If $\dot \gamma_j(0)$ is not a null-vector, then
 \[
    \Psi_{\gamma_j}=\sign(\dot v_j(0)) 
    \frac{\psi_p(\partial/\partial v)}{%
                 |\psi_p(\partial/\partial v)|}
    \qquad
    \left(\gamma_j(t)=\bigl(u_j(t),v_j(t)\bigr)\right).
 \]
 On the other hand,
 if $\gamma_j(t)$ is a singular curve with the null vector 
 $\dot \gamma_j(0)$,
 as seen in the proof of Proposition~\ref{prop:limit}, we get
 \[
   \Psi_{\gamma_j}=\left( \lim_{t\to+0} \sign\dot v_j(t) \right)
   \frac{\psi_p(\partial/\partial v)}{%
        |\psi_p(\partial/\partial v)|}.
 \]
 These two formulas for the $\E$-initial vector 
 $\Psi_{\gamma_j}$ prove the assertion.
\end{proof}

\begin{corollary}\label{cor:angle}
 Let $(U;u,v)$ be a $g$-coordinate system at a peak $p$, and
 $\Omega$ a singular sector. 
 Then the interior angle $\alpha_\Omega^{}$ of $\Omega$ 
 {\rm(}defined in \eqref{eq:interior-angle}{\rm )} is given by
 \[
    \alpha_\Omega^{}
        =\begin{cases}
       2\pi & \mbox{
	       \begin{tabular}[t]{l}
		 if $\Omega\cup\{p\}$ contains the closed upper-half
		 or the closed lower-half\\ $uv$-plane near $p$,
               \end{tabular}
	      }\\
       0    & \mbox{
	       \begin{tabular}[t]{l}
                 if $\Omega$ is contained in 
		 the open upper-half
		 or the open lower-half\\ $uv$-plane, 
               \end{tabular}
	      }\\
       \pi  &\mbox{
	      \begin{tabular}{l}
	       otherwise.
	      \end{tabular}
	     }
  \end{cases}
 \]
 See Figure \ref{fig:zero-pi}.
\end{corollary}

\begin{figure}
 \begin{center}
  \begin{tabular}{c@{\hspace{1cm}}c}
  \includegraphics{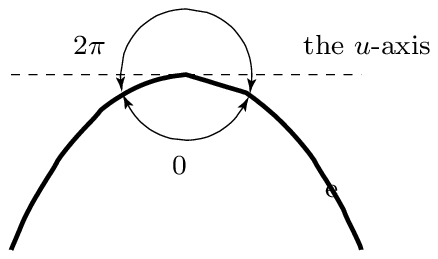} &
  \includegraphics{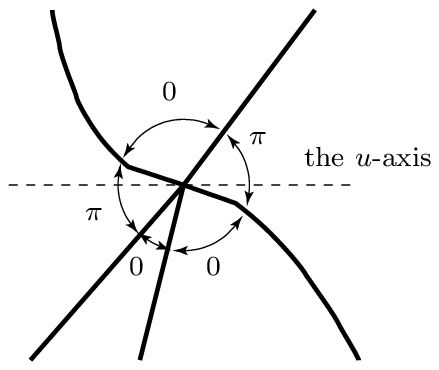} 
  \end{tabular}
 \end{center}
 \caption{Interior angles}
 \label{fig:zero-pi}
\end{figure}
Using this corollary, we can easily compute the angle of each singular
sector at peak:
For example, the singular set near the peak of 
the swallowtail $f_{\SW}$ given
\eqref{eq:cuspidal-swallow} is a parabola in the $uv$-plane. 
So both of the singular curves
starting at the origin are in the same class.
If we take the unit normal vector
$\nu_{\SW}=(1,u,u^2)/\sqrt{1+u^2+u^4}$ of $f_{\SW}$,
then the positive sector is the upper half domain of the parabola
and its interior angle is $2\pi$.

On the other hand, near the peak of the double swallowtail
as in Example~\ref{ex:double-swallow},
the singular set can be taken to be two lines
transversally intersecting at the origin, 
i.e.\ consisting of four rays.
Since the null direction is $\partial/\partial u$ on the singular set,
they are divided into two classes consisting of two rays.
The interior angles of positive sectors are both zero
and the negative sectors are both $\pi$ 
(see Figure~\ref{fig:double-swallow}).

The case of an isolated peak as in Example~\ref{ex:cuspidal-lip},
namely a cuspidal lips,
the neighborhood of the origin is the positive or negative sector,
and the interior angle is $2\pi$.

\begin{proof}[Proof of Theorem~\ref{thm:A}]
 The singular curves starting at the peak divide the neighborhood of $p$ 
 into sectors consisting of subsets of $M_+$ or $M_-$.
 However, by Corollary~\ref{cor:angle},
 there is no contribution of the interior angle of the sector
 unless it contains the $u$-axis to the $g$-coordinate system,
 that is, the only two sectors containing the $u$-axis
 have interior angle $\pi$.
 Thus we have
 \[
    \alpha_++\alpha_-=2\pi
 \]
 and
 \[
   \alpha_-,\alpha_+\in \bigl\{0,\pi,2\pi\bigr\},
 \]
 which proves the assertion.
\end{proof}

\section{A local Gauss-Bonnet formula}
\label{sec:localGB}

In this section, we state a local Gauss-Bonnet type theorem 
for ``admissible'' triangles.
Let $p$ be a peak and fix a $g$-coordinate system $(U;u,v)$ at $p$.
Let $\sigma(t)$ be an admissible curve 
(in the sense of Definition~\ref{def:admissible-curve}).
We define the {\em geometric curvature\/} $\tilde \kappa_g$
of $\sigma$ as follows:
\[
  \tilde \kappa_g(t)=
  \begin{cases}
   \hphantom{-}
   \hat \kappa_g(t) \qquad& (\mbox{if $\sigma(t)\in M_+$}), \\
   -\hat \kappa_g(t)\qquad& (\mbox{if $\sigma(t)\in M_-$}), \\
   \hphantom{-}
   \kappa_s(t) \qquad& (\mbox{if $\sigma(t)\in \Sigma$}).
  \end{cases}
\]
Here, this (geometric) curvature $\tilde \kappa_g$ is 
the geodesic curvature with respect to the orientation of $M^2$
which coincides with the curvature $\hat \kappa_g$  
defined by \eqref{eq:geodesic-curvature} on $M_+$ and is equal to 
$-\hat \kappa_g$ on $M_-$.

\begin{definition}[Admissible triangles]\label{def:admissible-triangle}
 Let $\overline T\subset U$ be the closure of a  simply connected 
 domain $T$ 
 which is bounded by three admissible arcs $\gamma_1$, $\gamma_2$,
 $\gamma_3$. 
 Let $\pt{A}$, $\pt{B}$, and $\pt{C}$ be the distinct three boundary
 points of $T$ which are intersections of these three arcs.
 Then $\overline T$ is called an {\em admissible triangle\/} if
 it satisfies the following three conditions:
\begingroup
\renewcommand{\theenumi}{(\alph{enumi})}
\renewcommand{\labelenumi}{\theenumi}
\begin{enumerate}
 \item\label{item:a:triangle}
         $\overline T$ admits at most one peak on
         $\{\pt{A},\pt{B},\pt{C}\}$.
 \item\label{item:b:triangle}
         The three interior angles 
at $\pt{A}$, $\pt{B}$, and $\pt{C}$ with
         respect to the metric $g$ are all less than $\pi$.
 \item\label{item:c:triangle}
         If $\gamma_j$ ($j=1,2,3$) is not a singular curve,
         it is $C^2$-regular, namely it
         is a restriction of a certain open $C^2$-regular arc.
\end{enumerate}
\endgroup
 We write
 \[
    \triangle \pt{ABC}:=\overline T
 \]
 and call $\{\pt{A},\pt{B},\pt{C}\}$ the {\em vertices\/} of the triangle.
 We also denote by
 \[
   \overset{\frown}{\pt{BC}}:=\gamma_1,\qquad
   \overset{\frown}{\pt{CA}}:=\gamma_2\quad \text{and}\quad
   \overset{\frown}{\pt{AB}}:=\gamma_3
 \]
 the regular arcs whose boundary points are $\{\pt{B},\pt{C}\}$,
 $\{\pt{C},\pt{A}\}$, and $\{\pt{A},\pt{B}\}$, respectively.
 We give here the orientation of these three arcs such that
 the left-hand side is $T$,
 namely the cyclic order $(\pt{A}, \pt{B}, \pt{C})$ is
 compatible with respect to 
 the orientation of $M^2$, see Figure~\ref{fig:admissible}.
\end{definition}

\begin{figure}
 \begin{center}
    \includegraphics{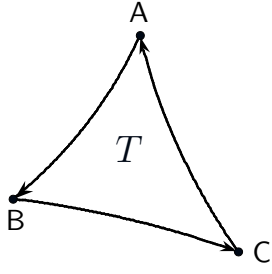}
 \end{center}
\caption{An admissible triangle}
\label{fig:admissible}
\end{figure}

We also denote by 
\[
  \angle\pt{A},\qquad \angle\pt{B},\quad\text{and}\quad \angle\pt{C}
\]
the {\em interior angles\/}
(with respect to the first fundamental form 
 $ds^2=\psi^*\inner{~}{~}$)
of the piecewise smooth boundary of $\triangle\pt{ABC}$ at 
$\pt{A}$, $\pt{B}$, and $\pt{C}$,
respectively unless $\pt{A}$, $\pt{B}$ and $\pt{C}$
are not singular points.
On the other hand, if $\pt{A}$ is a singular point, 
we set
\begin{equation}\label{def:lim-angle}
 \angle \pt{A}:=
  \begin{cases}
   \pi \qquad& 
     \mbox{
       \begin{tabular}[t]{l}
        if the $u$-curve passing through $\pt{A}$ separates\\
        $\overset{\frown}{\pt{AB}}$ and 
        $\overset{\frown}{\pt{AC}}$,
       \end{tabular}}
   \\
   0   \qquad & \mbox{\begin{tabular}{l}otherwise.\end{tabular}}
  \end{cases}
\end{equation}
Similarly we can define $\angle \pt{B}$ (resp. $\angle \pt{C}$)
when $\pt{B}$ (resp. $\pt{C}$) is a singular point.

\begin{remark}
 By Proposition~\ref{prop:g-angle}, $\angle \pt{A}$ coincides with 
 the angle of the $\E$-initial vectors between
 $\overset{\frown}{\pt{AB}}$ and 
 $\overset{\frown}{\pt{AC}}$.
\end{remark}

\begin{theorem}[The local Gauss-Bonnet formula]%
\label{thm:main}
 Let $(\E,\inner{~}{~},D,\psi)$ be a coherent tangent bundle and
 $\triangle\pt{ABC}$ an admissible triangle
 on $M^2$.
 Then the following identity holds{\rm :}
 \begin{equation}\label{eq:locGB0}
  \angle\pt{A}+\angle\pt{B}+\angle\pt{C}-\pi=
   \int_{\partial \triangle\pt{ABC}} \tilde \kappa_g\, d\tau
     +
     \int_{\triangle\pt{ABC}} K\, dA
     +
     2
     \int_{\Sigma\cap (\triangle\pt{ABC})^\circ}\! \kappa_s\, d\tau,
 \end{equation}
 where $\Sigma$ is the singular set,
 $(\triangle\pt{ABC})^\circ$ 
 {\rm(}resp.\  $\partial\triangle\pt{ABC}${\rm)} 
 the interior {\rm(}the boundary{\rm)} of 
 the closed domain $\triangle\pt{ABC}$, 
 and $K$ is the Gaussian curvature of 
 the metric $ds^2=\psi^*\inner{~}{~}$ on  $M^2\setminus \Sigma$.
 In particular, if there are no singular points  in the interior of the
 triangle, it holds that
 \begin{equation}\label{eq:locGB}
  \angle\pt{A}+\angle\pt{B}+\angle\pt{C}-\pi=
   \int_{\partial \triangle\pt{ABC}} \tilde \kappa_g d\tau
   +
   \int_{\triangle\pt{ABC}} K\, dA.
 \end{equation}
\end{theorem}

To prove Theorem~\ref{thm:main},
we prepare several lemmas as follows:
\begin{lemma}\label{lem1}
 Suppose that $\triangle\pt{ABC}$ is contained in $M_+$ or $M_-$.
 Then \eqref{eq:locGB} holds.
\end{lemma}
\begin{proof}
 The lemma is exactly the classical Gauss-Bonnet formula
 with respect to the Riemannian metric $ds^2$
 on $M^2\setminus \Sigma$.
\end{proof}
\begin{figure}
\begin{center}
\footnotesize
\begin{tabular}{c@{\hspace{10mm}}c}
        \includegraphics{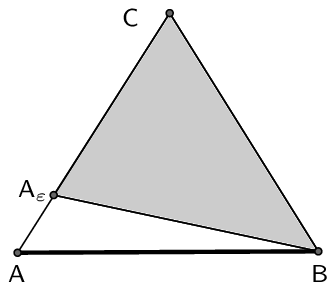} &
        \includegraphics{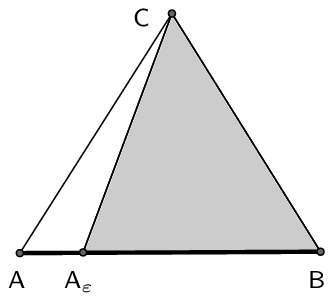} 
 \\
 Proof of Lemmas~\ref{lem2} and \ref{lem3} &
 Proof of Lemma~\ref{lem4}
\end{tabular}
\end{center}
\caption{%
 Proofs of Lemmas~\ref{lem2}, \ref{lem3} and \ref{lem4}.
}
\label{fig:proofs1}
\end{figure}
    
\begin{lemma}\label{lem2}
 Let $\triangle\pt{ABC}$ be an admissible triangle 
 such that $\pt{A}$ is a $A_2$-point or a peak, and 
 $\triangle\pt{ABC}\setminus \{\pt{A}\}$ 
 lies  in $M_+$  {\rm (}resp.\ $M_-${\rm)}.
 Suppose that $\overset{\frown}{\pt{AB}}$
 and $\overset{\frown}{\pt{BC}}$ are transversal
 at $\pt{B}$.  Then \eqref{eq:locGB} holds.
\end{lemma}
\begin{proof}
 Without loss of generality, we may assume that 
 $\triangle\pt{ABC}\setminus \{\pt{A}\}$ lies in $M_+$.
 We can take a short extension of 
 the $C^2$-regular arc $\overset{\frown}{\pt{AB}}$ beyond $\pt{A}$, and 
 rotate it around $\pt{B}$
 with respect to the canonical metric $du^2+dv^2$
 on the $uv$-plane. 
 Then we get 
 a smooth $1$-parameter family of $C^2$-regular arcs staring at $\pt{B}$. 
 Since $\overset{\frown}{\pt{AB}}$ and 
 $\overset{\frown}{\pt{BC}}$ are transversal at $\pt{B}$,
 restricting the image of this family 
 on the triangle $\triangle\pt{ABC}$,
 we get a family of $C^2$-regular curves
 \[
    \gamma_\varepsilon:[0,1]\longrightarrow \triangle\pt{ABC}
        \qquad (\varepsilon\in [0,1])
 \]
 such that (see Figure~\ref{fig:proofs1}, left)
 \begingroup
 \renewcommand{\theenumi}{(\roman{enumi})}
 \begin{enumerate}
  \item $\gamma_0$ parametrizes $\overset{\frown}{\pt{AB}}$
    such that $\gamma_0(1)=\pt{A}$ and $\gamma_0(0)=\pt{B}$,
  \item $\gamma_\varepsilon(0)=\pt{B}$ for all $\varepsilon\in [0,1]$,
  \item the correspondence 
    $\sigma:\varepsilon\mapsto\gamma_\varepsilon(1)$
    gives a subarc on $\overset{\frown}{\pt{AC}}$.
    We set $\pt{A}_\varepsilon=\gamma_\varepsilon(1)$, 
    where $\pt{A}_0=\pt{A}$.
 \end{enumerate}
 \endgroup
 Since 
 $\triangle\pt{A}_\varepsilon\pt{BC}$ ($\varepsilon>0$)
 lies in $M_+$, it is an admissible triangle.
 So, applying Lemma~\ref{lem1} for $\triangle\pt{A}_\varepsilon\pt{BC}$, 
 we have
 \[
    \angle\pt{A}_\varepsilon+
    \angle\pt{A}_\varepsilon\pt{BC}+\angle\pt{C}-\pi=
    \int_{\partial \triangle \pt{A}_\varepsilon\pt{BC}} 
      \tilde\kappa_g\, d\tau + 
    \int_{\triangle\pt{A}_\varepsilon\pt{BC}} K\, dA.
 \]
 By taking the limit as $\varepsilon\to 0$, we have that
 \[
   \lim_{\varepsilon\to +0}\angle\pt{A}_\varepsilon 
     +\angle\pt{B}+\angle\pt{C} -\pi=
   \int_{\partial \triangle\pt{ABC}} \tilde\kappa_g\, d\tau + 
    \int_{\triangle\pt{ABC}} K\, dA.
 \]
 Note that since $\triangle\pt{ABC}$ is admissible, 
 $\hat\kappa_g$ is bounded on both of $\overset{\frown}{\pt{AB}}$ and 
 $\overset{\frown}{\pt{AC}}$.
 On the other hand, by Proposition~\ref{prop:g-angle}
 we have
 \begin{align*}
  \lim_{\varepsilon\to +0}\angle\pt{A}_\varepsilon&=
  \lim_{\varepsilon\to +0}
     \arccos
        \left[
        \left.
         \inner{%
            \dot{\hat\gamma}_{\varepsilon}(0)
          }{%
            \psi\left(\frac{d\gamma_{s}(0)}{ds}\right)
          }\right/
         \left(
         \left|
            \dot{\hat\gamma}_{\varepsilon}(0)
         \right|
         \left|
            \psi\left(\frac{d\gamma_{s}(0)}{ds}\right)
         \right|\right)
        \right]_{s=\varepsilon}
    \\
    &=\arccos\inner{\Psi_{\gamma_0}}{\Psi_{\sigma}}
    =
  \begin{cases}
   \pi \quad&  \mbox{%
                \begin{tabular}[t]{l}
		 if the $u$-curve passing through $\pt{A}$\\
		 separates
		 $\overset{\frown}{\pt{AB}}$
                 and $\overset{\frown}{\pt{AC}}$,
		\end{tabular}}
   \\
   0   \quad&  \mbox{\begin{tabular}{l}otherwise,\end{tabular}}
  \end{cases}
 \end{align*}
 where $\dot{\hat\gamma}_{\varepsilon}(t)=
  \psi\bigl(d\gamma_\varepsilon(t)/dt\bigr)$.
 This completes the proof.
\end{proof}

\begin{lemma}\label{lem3}
 Let $\triangle\pt{ABC}$ be an admissible triangle such 
 that $\overset{\frown}{\pt{AB}}$ consists of  $A_2$-points, and
 $\triangle\pt{ABC} \setminus \overset{\frown}{\pt{AB}}$
 lies  in $M_+$  {\rm(}resp.\ in $M_-${\rm )}.
 Suppose that $\overset{\frown}{\pt{AB}}$
 and $\overset{\frown}{\pt{BC}}$ are transversal
 at $\pt{B}$. 
 Then \eqref{eq:locGB} holds.
\end{lemma}
\begin{proof}
 Without loss of generality, we may assume that
 $\triangle\pt{ABC}\setminus \overset{\frown}{\pt{AB}}$ lies in $M_+$.
 Since $\pt{A}$ and $\pt{B}$ are $A_2$-points,
 by the same method as in Lemma \ref{lem2}
there is a family of $C^2$-regular curves
 \[
   \gamma_\varepsilon:[0,1]\longrightarrow
                 \triangle\pt{ABC}\qquad (\varepsilon\in [0,1])
 \]
 such that
 \begingroup
 \renewcommand{\theenumi}{(\roman{enumi})}
 \begin{enumerate}
  \item $\gamma_0$ parametrizes $\overset{\frown}{\pt{AB}}$
    such that $\gamma_0(1)=\pt{A}$ and $\gamma_0(0)=\pt{B}$,
  \item $\gamma_\varepsilon(0)=\pt{B}$ for all $\varepsilon\in [0,1]$,
  \item the correspondence $\varepsilon\mapsto \gamma_\varepsilon(1)$ gives 
    a subarc on $\overset{\frown}{\pt{AC}}$.
    We set $\pt{A}_\varepsilon=\gamma_\varepsilon(1)$, 
    where $\pt{A}_0=\pt{A}$.
 \end{enumerate}
 \endgroup
 Since 
 $\triangle\pt{A}_\varepsilon\pt{BC}$ ($\varepsilon>0$)
 lies in $M_+$, it is an admissible triangle.
 So, applying Lemma~\ref{lem2} for $\triangle \pt{A}_\varepsilon \pt{BC}$, 
 we have
 \[
    \angle\pt{A}_\varepsilon+\angle\pt{A}_\varepsilon\pt{BC}
       +\angle\pt{C}-\pi=
    \int_{\partial \triangle\pt{A}_\varepsilon\pt{BC}} 
     \tilde\kappa_g \,d\tau + 
     \int_{\triangle\pt{A}_\varepsilon\pt{BC}} K\, dA.
 \]
 By taking the limit as $\varepsilon\to 0$, we have 
 the assertion just by the same argument as in the proof of the previous
 lemma.
\end{proof}

\begin{lemma}\label{lem4}
 Let $\pt{A}$ be a peak of an admissible triangle $\triangle\pt{ABC}$ 
 such that $\overset{\frown}{\pt{AB}}$ is a singular curve 
 starting from  $\pt{A}$,
 and $\triangle\pt{ABC}\setminus \overset{\frown}{\pt{AB}}$
 is contained completely in $M_+$ or in $M_-$.
Suppose that $\overset{\frown}{\pt{BC}}$
and $\overset{\frown}{\pt{CA}}$ are transversal
at $\pt{C}$. 
Then \eqref{eq:locGB} holds.
\end{lemma}

\begin{proof}
 We may assume that
 $\triangle\pt{ABC}$ lies in $\overline{M}_+$.
 We can take a short extension of 
 the $C^2$-regular arc $\overset{\frown}{\pt{AC}}$ beyond
 $\pt{A}$, 
 and  rotate it around $\pt{C}$
 with respect to the canonical metric $du^2+dv^2$
 on the $uv$-plane. 
 Then we get 
 a smooth $1$-parameter family of 
 $C^2$-regular arcs staring at $\pt{C}$. 
 Since $\overset{\frown}{\pt{BC}}$ and 
 $\overset{\frown}{\pt{AC}}$ are transversal at $\pt{C}$,
 restricting the image of this family on the triangle $\triangle\pt{ABC}$,
 we get a family of $C^2$-regular curves
 \[
   \gamma_\varepsilon:[0,1]\longrightarrow 
       \triangle\pt{ABC},\qquad (\varepsilon\in [0,1])
 \]
 such that (see Figure~\ref{fig:proofs1}, right)
 \begingroup
 \renewcommand{\theenumi}{(\roman{enumi})}
 \begin{enumerate}
  \item $\gamma_0$ parametrizes $\overset{\frown}{\pt{AC}}$
	such that $\gamma_0(1)=\pt{A}$ and $\gamma_0(0)=\pt{C}$,
  \item $\gamma_\varepsilon(0)=\pt{C}$ for all $\varepsilon\in [0,1]$,
  \item the correspondence $\varepsilon\mapsto \gamma_\varepsilon(1)$ 
	gives 
	a subarc on $\overset{\frown}{\pt{AB}}$.
	We set $\pt{A}_\varepsilon=\gamma_\varepsilon(1)$, 
	where $\pt{A}_0=\pt{A}$.
 \end{enumerate}
 \endgroup
 Since $\overset{\frown}{\pt{AC}}$ is admissible,
 its tangential vector at $\pt{A}$ does not point in
 the null-direction.
 Hence, 
 the tangential vector of  $\overset{\frown}{\pt{A}_\varepsilon\pt{C}}$ at 
 $\pt{A}_\varepsilon$ does not point in
 the null-direction for sufficiently small
 $\varepsilon>0$ and 
 $\triangle\pt{A}_\varepsilon\pt{BC}$ is an admissible triangle.
 Applying Lemma~\ref{lem3} for $\triangle \pt{A}_\varepsilon\pt{BC}$, 
 we have
 \[
   \angle\pt{A}_\varepsilon +\angle\pt{A}_\varepsilon\pt{CB}+
   \angle \pt{B}-\pi=
        \int_{\partial \triangle \pt{A}_\varepsilon\pt{BC}} 
     \tilde\kappa_g\, d\tau + 
     \int_{\triangle\pt{A}_\varepsilon\pt{BC}} K\, dA.
 \]
 By Proposition~\ref{prop:bounded-curvature-measure} and
 the Lebesgue convergence theorem, the limit
 \[
   \int_{\partial \triangle\pt{ABC}}\tilde\kappa_g\, d\tau
     =\lim_{\varepsilon\to +0}
        \int_{\partial \triangle\pt{A}_\varepsilon\pt{BC}}
          \tilde\kappa_g\, d\tau
 \]
 exists.
 On the other hand, since the property that
 the $u$-direction $\partial /\partial u$ 
 at $\pt{A}_\varepsilon$
 is inward or outward is common in $\varepsilon\in (0,1]$
 (cf.\ the proof of Theorem~\ref{thm:class}),
 $\angle \pt{A}_\varepsilon$ is common in $\varepsilon$.
 Thus we have 
 \[
    \angle\pt{A}=\lim_{\varepsilon\to +0}
     \angle\pt{A}_\varepsilon=
    \begin{cases}
      \pi \quad&  
         \mbox{if the $u$-curve separates $\overset{\frown}{\pt{AB}}$
               and $\overset{\frown}{\pt{AC}}$},
     \\
     0 \quad& \mbox{otherwise}.
    \end{cases}
 \]
 This proves the assertion.
\end{proof}

\begin{figure}
\begin{center}
\footnotesize
 \includegraphics{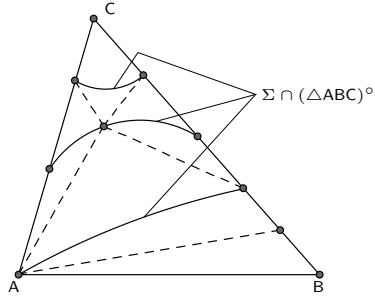}
\end{center}
\caption{%
  A proof of Lemma~\ref{lem5}.
}
\label{fig:proofs2}
\end{figure}

\begin{lemma}\label{lem5}
 Let $\pt{A}$ be a peak of an admissible triangle $\triangle\pt{ABC}$. 
 Suppose that there are at most one
 singular curve in 
 $\triangle\pt{ABC}$
 starting at $\pt{A}$ from the null direction.
 Then \eqref{eq:locGB} holds.
\end{lemma}
\begin{proof}
 By a suitable division,   
 the triangle decomposed into admissible triangles
 which satisfies the one of  the conditions as in
 Lemmas \ref{lem1}--\ref{lem4}.
 Then the formula is proved, since the geometric curvature on
 each edge consists of singular curve is duplicated.
 See Figure~\ref{fig:proofs2}.
\end{proof}

\begin{figure}
\begin{center}
\footnotesize
\begin{tabular}{c@{\hspace{10mm}}c}
        \includegraphics{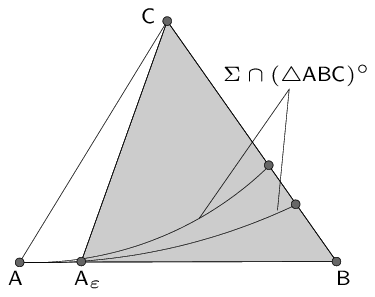} &
        \includegraphics{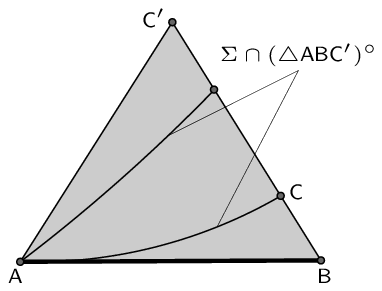} 
 \\
 Proofs of Lemmas~\ref{lem6} &
 Proof of Theorem~\ref{thm:main}
\end{tabular}
\end{center}
\caption{%
 Proofs of Lemmas~\ref{lem6} and Theorem~\ref{thm:main}
}
\label{fig:proofs3}
\end{figure}

\begin{lemma}\label{lem6}
 Let $\pt{A}$ be a peak of an admissible 
 triangle $\triangle\pt{ABC}$
such that
$\overset{\frown}{\pt{AB}}$
and $\overset{\frown}{\pt{BC}}$ are transversal
at $\pt{B}$. 
 Suppose that either $\overset{\frown}{\pt{AB}}$ or
 $\overset{\frown}{\pt{AC}}$ is not a singular curve.
 Then \eqref{eq:locGB} holds.
\end{lemma}

\begin{proof}
 The proof is almost parallel to
 that of Lemma \ref{lem4}
 (instead of Lemma \ref{lem3}, we apply Lemma \ref{lem5}):
 Without loss of generality, we may assume that
 $\overset{\frown}{\pt{AB}}$ is a singular curve.
 Then $\overset{\frown}{\pt{AC}}$ is a $C^2$-regular arc.
 We can take a short extension of 
 $\overset{\frown}{\pt{AC}}$ over $\pt{A}$, and 
 rotate it around $\pt{C}$
 with respect to the metric $du^2+dv^2$.
 Then we get 
 a smooth $1$-parameter family of arcs starting at $\pt{C}$. 
 Since $\overset{\frown}{\pt{BC}}$ and 
 $\overset{\frown}{\pt{AC}}$ are transversal at $\pt{C}$ by
the assumption of the lemma,
 restricting the image of this family to the triangle
 $\triangle\pt{ABC}$,
 we get a family of $C^2$-regular curves
 \[
   \gamma_\varepsilon:[0,1]\longrightarrow 
       \triangle\pt{ABC},\qquad (\varepsilon\in [0,1])
 \]
 such that
 \begingroup
 \renewcommand{\theenumi}{(\roman{enumi})}
 \begin{enumerate}
  \item $\gamma_0$ parametrizes $\overset{\frown}{\pt{AC}}$
	such that $\gamma_0(1)=\pt{A}$ and $\gamma_0(0)=\pt{C}$,
  \item $\gamma_\varepsilon(0)=\pt{C}$ for all $\varepsilon\in [0,1]$,
  \item the correspondence $\varepsilon\mapsto \gamma_\varepsilon(1)$ 
	gives 
	a subarc on $\overset{\frown}{\pt{AB}}$.
	We set $\pt{A}_\varepsilon=\gamma_\varepsilon(1)$, 
	where $\pt{A}_0=\pt{A}$.
 \end{enumerate}
 \endgroup
 Since $\overset{\frown}{\pt{AC}}$ is admissible,
 its tangential vector at $\pt{A}$ does not point in the null-direction.
 Hence, 
 the tangential vector of  $\overset{\frown}{\pt{A}_\varepsilon\pt{C}}$ at 
 $\pt{A}_\varepsilon$ does not point in the null-direction for 
 sufficiently small $\varepsilon>0$ and 
 $\triangle\pt{A}_\varepsilon\pt{BC}$ is an admissible triangle.
 Applying Lemma~\ref{lem5} for $\triangle \pt{A}_\varepsilon\pt{BC}$, 
 we have
 \[
   \angle\pt{A}_\varepsilon +\angle\pt{A}_\varepsilon\pt{CB}+
   \angle \pt{B}-\pi=
        \int_{\partial \triangle \pt{A}_\varepsilon\pt{BC}} 
    \tilde\kappa_g\, d\tau + 
     \int_{\triangle\pt{A}_\varepsilon\pt{BC}} K\, dA
    +2\int_{\Sigma\cap(\triangle \pt{A}_\varepsilon\pt{BC})^{\circ}}
             \!\!\kappa_s\,d\tau.
 \]
 By Proposition~\ref{prop:bounded-curvature-measure} and
 the Lebesgue convergence theorem, the limits
 \begin{align*}
   \int_{\partial \triangle\pt{ABC}}\tilde\kappa_g\, d\tau
     &=\lim_{\varepsilon\to +0}
        \int_{\partial \triangle\pt{A}_\varepsilon\pt{BC}}
          \tilde\kappa_g\, d\tau,\qquad\text{and}\\
         \int_{\Sigma\cap(\triangle\pt{ABC})^{\circ}}
          \kappa_s\,d\tau
     &
      =\lim_{\varepsilon\to +0}
      \int_{\Sigma\cap(\triangle\pt{A}_{\varepsilon}\pt{BC})^{\circ}}
          \kappa_s\,d\tau
 \end{align*}
 exist.
 On the other hand, since the property that
 the $u$-direction $\partial /\partial u$ 
 at $\pt{A}_\varepsilon$
 is inward or outward is common in $\varepsilon\in (0,1]$
 (cf.\ the proof of Theorem~\ref{thm:class}),
 $\angle \pt{A}_\varepsilon$ is common in $\varepsilon$.
 Thus we have 
 \[
    \angle\pt{A}=\lim_{\varepsilon\to +0}
     \angle\pt{A}_\varepsilon=
    \begin{cases}
      \pi \quad&  \mbox{if the $u$-curve separates $\overset{\frown}{\pt{AB}}$
                           and $\overset{\frown}{\pt{AC}}$},
     \\
     0 \quad& \mbox{otherwise}.
    \end{cases}
 \]
 This proves the assertion.
\end{proof}

\begin{proof}[Proof of Theorem \ref{thm:main}]
 As seen in the proof of Lemma~\ref{lem5},
 the given triangle can be divided into small 
 triangles.
 So it is sufficient to 
 consider the case that $\triangle\pt{ABC}$ has the 
 following four properties:
 \begin{enumerate}
  \item $\pt{A}$ is a peak,
  \item $\overset{\frown}{\pt{AB}}$ and $\overset{\frown}{\pt{AC}}$
    are both singular curves which have 
    the same null-velocity vector at $\pt{A}$,
  \item $\overset{\frown}{\pt{BC}}$ is not a singular curve,
  \item there are no singular points on the
    inside of the triangle $\triangle\pt{ABC}$.
 \end{enumerate}
 Take a new smooth arc
 $\overset{\frown}{\pt{BC}}{}'$ such that
 $\overset{\frown}{\pt{BC}}$ is a subarc.
 Then we consider a new admissible triangle
 $\triangle\pt{ABC}'$, which contains
 $\triangle\pt{ABC}$ as a subset, see Figure~\ref{fig:proofs3}, right.

 In this situation, the two  triangles 
 $\triangle\pt{ABC}'$  and $\triangle\pt{ACC}'$ satisfy 
 the assumption of Lemma~\ref{lem6}.
 So we have
 \begin{equation}\label{t1}
  \angle\pt{C}'\pt{AB} +\angle\pt{B}+\angle\pt{C}'-\pi=
   \int_{\partial \triangle \pt{ABC}'} \!\!\!
    \hat\kappa_g d\tau + 
     \int_{\triangle\pt{ABC}'} \!\!\!K\, dA
     +
     2 
     \int_{\Sigma\cap (\triangle\pt{ABC}')^\circ}\!\!\! \kappa_s\, d\tau,
 \end{equation}
 and 
 \begin{equation}\label{t2}
  \angle\pt{C}'\pt{AC} +\angle\pt{ACC}'+\angle\pt{C}'-\pi=
  \int_{\partial \triangle\pt{ACC}'} \!\!\!
  \hat\kappa_g\, d\tau + 
  \int_{\triangle\pt{ACC}'}\!\!\! K\, dA
  +2
  \int_{\Sigma\cap (\triangle\pt{ACC}')^\circ} \!\!\!\kappa_s\, d\tau.
  \end{equation}
 Subtracting \eqref{t2} from \eqref{t1}, we get
 \[
   \angle\pt{A} +\angle\pt{B}+\angle\pt{C}-\pi=
    \displaystyle\int_{\partial \triangle\pt{ABC}} 
    \hat\kappa_g d\tau + 
    \int_{\triangle\pt{ABC}} K\, dA +
    2\int_{\overset{\frown}{\pt{BC}}}\kappa_s\,d\tau,
 \]
 since 
 \[
    \int_{\Sigma\cap (\triangle\pt{ABC}')^\circ}\!\!\! \kappa_s\, d\tau
     -
     \int_{\Sigma\cap (\triangle\pt{ACC}')^\circ}\!\!\! \kappa_s\, d\tau
     =
    \int_{\overset{\frown}{\pt{BC}}} \kappa_s\, d\tau
 \]
 and $\angle\pt{ACC}'+\angle\pt{C}=\pi$.

 This proves the formula \eqref{eq:locGB0}
 for any admissible triangle.
\end{proof}

\section{The proof of Theorem~\ref{thm:B}.}
\label{sec:globalGB}

\begin{proof}[Proof of Theorem~\ref{thm:B}]
 Although $\partial M_+$ and $\partial M_-$ are the same set, 
 their orientations are opposite.
 However, the singular curvature $\kappa_s$ 
 does not depend on the orientation of the singular curve.
 So we have
 \begin{equation}\label{eq:integral-of-singular-curvature}
    \int_{\partial M_+}\!\! \kappa_s\, d\tau 
      +\int_{\partial M_-}\!\! \kappa_s\, d\tau=
    2\int_{\Sigma} \kappa_s\, d\tau.
 \end{equation}
 The singular set on a sufficiently small neighborhood of 
 a peak $p$
 consists of finitely many regular $C^1$-curves starting from $p$.
 (The number $2m(p)$ of these singular curves starting from $p$
 is always even.)
 Since the  integrations of geometric curvatures 
 not on singular curves are cancelled by opposite integrations,
 we have from  Theorem~\ref{thm:main} that
 \begin{align*}
   2\pi\chi(M_+)
   &=\int_{M_+} K\,dA+\int_{\partial M_+}\!\!\kappa_s\, d\tau +
         \sum_{p:\text{peak}}\bigl(m(p)\pi-\alpha_+(p)\bigr),\\
   2\pi\chi(M_-)
   &=\int_{M_-} K\,dA+\int_{\partial M_-}\!\!\kappa_s\, d\tau +
         \sum_{p:\text{peak}}\bigl(m(p)\pi-\alpha_-(p)\bigr).
 \end{align*}

 Hence by \eqref{eq:integral-of-singular-curvature}, \eqref{eq:sum}
 and Definition~\ref{def:positive}, we have
 \begin{align*}
  2\pi\bigl(\chi(M_+)+\chi(M_-)\bigr)
  &=  \int_{M^2} K\,dA +
                   2\int_{\Sigma}\kappa_s\,d\tau
  +2\pi\sum_{p:\text{peak}}(m(p)-1),\\
  2\pi\bigl(\chi(M_+)-\chi(M_-)\bigr) 
                   &= \int_{M^2} K\,d\hat A
                         -
                     \sum_{p:\text{peak}}
                      \bigl(\alpha_+(p)-\alpha_-(p)\bigr).
 \end{align*}
 Since
 $M^2$ is the disjoint union of $M_+$, $M_-$ and $\Sigma$, the 
 following 
 formula for Euler numbers holds:
 \[
 \chi(M^2)=\chi(M_+)+\chi(M_-)+\chi(\Sigma).
 \]
 Since we assumed that $\Sigma$ admits at most peaks,
 $\Sigma$ is a finite topological graph.
 Hence we have $\chi(\Sigma)=\sum_{p:\text{peak}}(1-m(p))$.
 Thus, we have
 \begin{align*}
  2\pi\bigl(\chi(M_+)+\chi(M_-)+\chi(\Sigma)\bigr)
  &=  \int_{M^2} K\,dA +
                   2\int_{\Sigma}\kappa_s\,d\tau,\\
  \chi(M_+)-\chi(M_-) 
                   &=\frac{1}{2\pi} \int_{M^2}K d\hat A
                         -
  (\# P_+-\# P_-),
 \end{align*}
 where we used \eqref{eq:diff}. 
 Thus we have \eqref{eq:A}.
 Finally, by \eqref{eq:d-omega} and \eqref{eq:Euler}, we have 
 \eqref{eq:B}.
\end{proof}


\end{document}